\documentclass[12pt, twoside]{article}
\usepackage{amsmath,amsthm,amssymb}
\usepackage{times}
\usepackage{enumerate}

\def\titlerunning#1{\gdef\titrun{#1}}
\makeatletter
\def\author#1{\gdef\autrun{\def\and{\unskip, }#1}\gdef\@author{#1}}
\def\address#1{{\def\and{\\\hspace*{18pt}}\renewcommand{\thefootnote}{}%
\footnote {#1}}%
\markboth{\autrun}{\titrun}}
\makeatother
\def\email#1{e-mail: #1}
\def\subjclass#1{{\renewcommand{\thefootnote}{}%
\footnote{{Mathematics Subject Classification. } #1}}}
\def\keywords#1{\par\medskip
\noindent\textbf{Keywords.} #1}

\newtheorem{theorem}{Theorem}[section]
\newtheorem{corollary}[theorem]{Corollary}
\newtheorem{lemma}[theorem]{Lemma}

\theoremstyle{definition}
\newtheorem{defin}[theorem]{Definition}

\newtheorem{exa}[theorem]{Example}

\newtheorem*{xrem}{Remark}

\numberwithin{equation}{section}

\let\ep=\varepsilon
\newcommand{\vp}{\varphi}

\begin{document}

\baselineskip=16pt


\titlerunning{Siegel domains over Finsler symmetric cones}

\title{Siegel domains over Finsler symmetric cones}

\author{Cho-Ho Chu}

\date{}

\maketitle

\address{C-H. Chu: School of Mathematical Sciences, Queen Mary, University of London, London
E1 4NS, U.K.; \email{c.chu@qmul.ac.uk}}



\begin{abstract}
Let $\Omega$ be a proper open cone in a real Banach space $V$. We show that the tube domain
$V \oplus i\Omega$ over $\Omega$ is biholomorphic to a bounded symmetric domain if and only if
$\Omega$ is a normal linearly homogeneous Finsler symmetric cone, which is equivalent to the condition that
$V$ is a unital JB-algebra in an equivalent norm and $\Omega$ is the interior of
$\{v^2: v\in V\}$.
\end{abstract}


\section{Introduction}\label{sec1}

Let $V \oplus i\Omega$ be a Siegel domain of the first kind over a proper open cone $\Omega$ in a real Banach space $V,$ often called a {\it tube domain}.
If $V$ is finite dimensional, it is well-known from the seminal  works of
Koecher \cite{k} and Vinberg \cite{v} that $V \oplus i\Omega$ is biholomorphic to a bounded
symmetric domain  if and only if $\Omega$ is
a linearly homogeneous self-dual cone,  or equivalently, the closure $\overline\Omega$ is the
 cone $\{a^2: a\in  \mathcal{A}\}$ in  a formally real Jordan algebra $\mathcal{A}$,
 in which case $\Omega$  carries the structure of a Riemannian symmetric space (see also \cite{ash, fk, satake}).
 This result has an infinite
dimensional extension by the work of Braun, Kaup and Upmeier in \cite{bku,ku}, which shows
that $V \oplus i\Omega$ of any dimension is biholomorphic to
a bounded symmetric domain if and only if $\overline\Omega=\{a^2: a\in  \mathcal{A}\}$
in a unital JB-algebra $ \mathcal{A}$. In both cases, $V$ is the underlying vector space of $\mathcal{A}$.
If moreover, $V$ is a Hilbert space, then $\Omega$ is
also a Riemannian symmetric space \cite{book}. However, in contrast to the finite dimensional case, the question of
characterising tube domains  $V \oplus i\Omega$  which are biholomorphic to a bounded symmetric domain
in terms of the geometric structure  of $\Omega$ has been open. The question amounts to extending Koecher and Vinberg's condition of a linearly homogeneous self-dual cone to infinite dimensional Banach spaces.
A fundamental obstacle is that the concept of
a self-dual cone is unavailable in infinite dimensional Banach spaces from want of a positive definite
quadratic form. Nevertheless, using the Finsler structure, we are able to circumvent this difficulty and address
the above question positively.

We show that the tube domain $V \oplus i\Omega$ is biholomorphic to a bounded
symmetric domain  if and only if $\Omega$ is a normal linearly homogeneous Finsler symmetric cone. The latter can be viewed as an
infinite dimensional generalisation of the notion of a linearly homogeneous self-dual cone. Further details are given below.

Let $\Omega$ be an open cone in a real Banach space $V$. Then $\Omega$ is a real  Banach manifold
modelled on $V$.
Let $L(V)$ be
the Banach algebra of bounded linear operators on $V,$ which is a real Banach Lie algebra in the Lie brackets
 $$[S,T] := ST-TS \qquad (S,T\in L(V)).$$
Let $GL(V)$ be the open subgroup of $L(V)$ consisting of invertible elements in $L(V)$.
It is a real Banach Lie group with Lie algebra $L(V)$. The linear maps $g\in GL(V)$ satisfying
$g(\Omega) = \Omega$
 form a subgroup of $GL(V)$ and will be denoted by
 \begin{equation}\label{G}
G(\Omega)=\{g\in GL(V): g(\Omega) = \Omega\}.
\end{equation}
We shall call $G(\Omega)$ the {\it linear automorphism group} of $\Omega$.
An element $g\in GL(V)$ belongs to
$G(\Omega)$ if and only if $g(\overline\Omega) = \overline\Omega$, the latter denotes the closure of $\Omega$.
 Hence $G(\Omega)$ is a closed subgroup of
 $GL(V)$ and can be topologised  to
 a real Banach Lie group with Lie algebra
 \begin{equation}\label{go}
\frak g(\Omega) =\{ X\in L(V): \exp t X \in G(\Omega),  \forall t \in \mathbb{R}\}
\end{equation}
(cf.\,\cite[p.\,387]{up}).

 An open cone $\Omega$  in $V$ can be homogeneous under various group actions.
The terminology  {\it linear homogeneity} throughout the paper is defined below.
\begin{defin} An open cone $\Omega$ in a real Banach space
is called {\it linearly homogeneous} if the linear automorphism
group $G(\Omega)$ acts transitively on $\Omega$,
that is, given $a,b\in \Omega$, there is a continuous linear isomorphism $g\in G(\Omega)$ such that
$g(a)=b$.
\end{defin}

An open cone $\Omega$ in a real Hilbert space  $V$  with an inner product
$\langle \cdot,\cdot\rangle$ is called {\it self-dual} if $\Omega=\Omega^*$, where
$$\Omega^*= \{v\in V: \langle v,x\rangle >0, \forall x\in \overline\Omega \backslash\{0\}\}$$
denotes the dual cone of $\Omega$.

\begin{xrem} Linearly homogeneous self-dual cones are often called {\it symmetric cones} in literature.
In this paper, we adopt the former terminology to avoid the latter being confused with the notion
of {\it symmetric domains}.
\end{xrem}

Recently, the  result of  Koecher \cite{k} and Vinberg \cite{v} has been extended to infinite dimensional
Hilbert spaces in \cite{chu},
where it has been shown that an open cone $\Omega$ in a real Hilbert space  $V$, with inner product $\langle\cdot,\cdot\rangle$,
is a linearly homogeneous self-dual cone if and only if
$V$ carries the structure of a Jordan algebra with identity
and $\overline\Omega=\{x^2:x\in V\}$, in which
the Jordan product satisfies
$$\langle ab,c\rangle = \langle b,ac\rangle \qquad (a,b,c \in V).$$
Such a real Jordan algebra, with or without identity, is called a {\it JH-algebra}.
Together with the  result of \cite{bku} mentioned before,
the above assertion implies that the tube domain $V \oplus i\Omega$ over an open cone $\Omega$ in a Hilbert space
$V$ is biholomorphic to a bounded symmetric domain if and only if $\Omega$ is linearly homogeneous and self-dual.

The question of extending this result to Banach spaces is a natural one although it has been unknown
what should be an appropriate generalisation of the concept of self-duality, which  is unavailable in arbitrary
Banach spaces.
In finite dimensional Euclidean spaces, it has been shown by Shima \cite{shima} and Tsuji \cite{tsuji}
 that if an open cone
$\Omega $ is linearly homogeneous,  and if $\Omega$ is a symmetric space in some
Riemannian metric,
 then it is self-dual and hence $V \oplus i\Omega$ is indeed biholomorphic to
a bounded symmetric domain.

In the absence of Riemannian structures and self-duality in Banach spaces,
we establish an equivalent geometric condition on $\Omega$
for  $V \oplus i\Omega$ to be biholomorphic to a bounded symmetric domain for Banach spaces $V,$
namely, that $\Omega$ be a normal linearly homogeneous Finsler symmetric cone.

\begin{defin}\label{finslm}
By a {\it Finsler symmetric cone}, we mean an open cone
 $\Omega$   in a real Banach space, which is a symmetric Banach manifold in a
{\it compatible $G(\Omega)$-invariant
tangent norm}  (defined in Section \ref{fins}).
\end{defin}
Normal cones are defined in Section \ref{jordanorder}. In finite dimensions,  proper open cones are normal.
Self-dual cones in Hilbert spaces are also normal.
We prove the following main result in Theorem \ref{main}, which resolves the aforementioned question.\\

\noindent {\bf Main Theorem.} 
{\it Let $\Omega$ be a proper open cone in a real Banach space $V.$
The following conditions are equivalent.
\begin{enumerate}[\upshape (i)]
\item The Siegel domain $V \oplus i\Omega$ is biholomorphic to a bounded symmetric domain.
\vspace{-.1in}
\item $\Omega$ is a normal linearly homogeneous Finsler symmetric cone.

\end{enumerate}}

Condition (ii) in this theorem also provides a simple order-geometric characterisation of unital JB-algebras
as it is equivalent to $V$ being a unital JB-algebra in an equivalent norm
and $\Omega$ the interior of  $\{a^2: a\in V\}$.
Hence Finsler symmetric cones abound.
The well-known characterisation of unital JB-algebras
by geometric properties of the state space has been achieved by Alfsen and Schultz in \cite{as},
which is the culmination of a  noncommutative spectral theory developed in a
series of papers \cite{as1,as2, as3}.

To prove the Main Theorem, we first give, in the next two sections,
the definition of symmetric Banach manifolds and JB-algebras, together
with some relevant results on cones and hermitian operators,
which will be used, in tandem with Jordan and Lie theory,
to establish the theorem in the last section.

\section{Symmetric Banach manifolds}\label{fins}

Let $M$ be a Banach manifold (with an analytic structure), modelled on a real or complex Banach space $(V, \|\cdot\|_{_V})$,
with  tangent bundle
$TM = \{(p,v): p\in M, v \in T_p M\}$. A  mapping
$$\nu : TM \longrightarrow [0,\infty)$$ is called a {\it tangent norm} \index{tangent norm}
if $\nu(p, \cdot)$ is a norm on the tangent space $T_p M \approx V$ for each $p\in M$. We call $\nu$
a {\it compatible tangent norm} if it satisfies
the following two conditions.
\begin{enumerate}
\item[(i)] $\nu$ is continuous.
\item[(ii)] For each $p\in M$, there is a local chart $\varphi: \mathcal{U} \rightarrow V$ at $p$, and
constants $0<r<R$ such that
$$ r\|d\varphi_a(v)\|_{_V} \leq \nu(a,v) \leq  R\|d\varphi_a(v)\|_{_V} \qquad (a\in \mathcal{U}\
\subset M, v\in T_a M).$$
\end{enumerate}
The integrated distance $d_\nu$ of the tangent norm $\nu$ on $M$ is given by
$$d_\nu(x,y) = \inf_\gamma \left\{\int_0^1 \nu(\gamma(t), \gamma'(t) )dt: \gamma(0)=x, \gamma(1)=y\right\}$$
where $\gamma :[0,1] \longrightarrow  M$ is a piecewise smooth curve in $M$.

\begin{xrem} In finite dimensions, a compatible tangent norm satisfying certain smoothness and convexity conditions
is known as a {\it Finsler metric} \cite{c}.  Nevertheless,
a Banach manifold with a compatible tangent norm is also
called a {\it Finsler manifold} in literature (e.g.\,\cite{neeb}) and this nomenclature
has been adopted in Definition \ref{finslm}.
\end{xrem}

Given a Banach manifold $M$ with a compatible tangent norm $\nu$,
a bianalytic map $f:M \longrightarrow M$ is called  a {\it $\nu$-isometry}  if it satisfies
\begin{equation*}\label{nuisom}
\nu(f(p), df_p(\cdot)) = \nu(p, \cdot) \quad {\rm for~all} \quad (p, \cdot) \in TM
\end{equation*}
in which case, we have $d_\nu(f(x),f(y)) = d_\nu (x,y)$ for all $x,y\in M$.

\begin{defin}\label{2.2}
Let $\Omega$ be an open cone in a real Banach space $V$, equipped with a tangent  norm $\nu$.
We say that $\nu$ is {\it $G(\Omega)$-invariant} if each $g \in G(\Omega)$ is a $\nu$-isometry.
\end{defin}

\begin{exa}\label{normanifold} A Riemannian manifold $(M,g)$ modelled on a real Hilbert space
$V$, with Riemannian metric $g$, admits a compatible tangent norm
$\nu : TM \longrightarrow [0,\infty)$
defined by
$$\nu(p, v):= g_p(v,v)^{1/2} \qquad  (p\in M, v \in T_pM \approx V).$$
The $\nu$-isometries of $M$ are exactly the isometries of $M$ with respect to the Riemannian metric $g$.
\end{exa}

\begin{exa}\label{cara}Let $D$ be a bounded domain in a complex Banach space $V$.
Then the Carath\'eodory
differential metric, defined below, is a compatible tangent norm on $D$.
$$\mathcal{C}(p,v)= \sup\{|f'(p)(v)|: f\in H(D, \mathbb{D}) \mbox{ and } f(p)=0\} \quad ( (p, v) \in TM)$$
where $H(D,\mathbb{D})$ is the set of all
holomorphic maps from $D$ to $\mathbb{D}=\{z\in \mathbb{C}: |z| <1 \}$. In this case,
all biholomorphic maps on $D$ are $\mathcal{C}$-isometries.
\end{exa}

An open cone $\Omega$ in a  real Banach space $V$ is a real connected  Banach manifold modelled on $V$.
A homogeneous polynomial $p: V \longrightarrow V$ of degree $n$ is of the form
$$p(v) = f(v,\ldots, v) \qquad (v\in V)$$
where  $f: V^{n}\longrightarrow V$ is a continuous $n$-linear map. In particular, each $f\in L(V)$ is a polynomial of
degree $1$, and polynomials of degree $0$ are the constant maps on $V$.

To each homogeneous polynomial
 $p$ on $V$, we associate an analytic vector field $p\frac{\partial}{\partial x}$ on $V$.
If $X= h\frac{\partial}{\partial x}$ is a linear vector field on $\Omega$, that is,
  $h$ is (the restriction of) a continuous linear map $f\in L(V)$, we identify $X$ with $f$.
Conversely, each $f\in L(V)$ identifies with the vector field $f\frac{\partial}{\partial x}$ on $\Omega$.

 Let $I\in L(V)$ be the identity map.
If  $X$ is a linear vector field on $\Omega$,  then evidently $[I,X]=0$. The converse is also true.
We sketch a proof for completeness.
Let $X = h\frac{\partial}{\partial x}$ be an analytic vector field and $[I,X]=0$, and let
$$h(x) = \sum_{n= -1}^\infty p_{_n}(x-e)$$
be the power series expansion of $h$  in a neighbourhood of a point $e\in \Omega$, where
$p_{_n}(v) = f_{_n}(v, \ldots, v)$ is a homogeneous polynomial of
degree $n+1$ with  $f_{_n}:
V^{n+1}\longrightarrow V$, and $p_{_{-1}}=h(e)$.  We have $$X= \sum_{n=-1}^\infty X_n,
\quad X_n =p_{_n}(x-e) \frac{\partial}{\partial x}$$ in a local chart at $e$  and
\begin{equation}
0= [I, X]= \sum_{n=-1}^\infty ({\rm ad}\,I)X_n  = \sum_{n=-1}^\infty
q_n\frac{\partial}{\partial x}.
\end{equation} implies
\begin{equation}\label{0}
\sum_{n=-1}^\infty
q_n(x) =0
\end{equation}
where $q_{_{-1}} = -h(e)$, $q_{_0}(x) = p_{_0}(e)$ and $q_{_1}(x) = f_{_1}(x-e,x) + f_{_1}(x,x-e)-p_{_1}(x-e)$.
This gives $-h(e)+p_{_0}(e)=0$ and
$$h(x) = p_{_0}(x) + p_{_1}(x-e) + \cdots.$$
Differentiating (\ref{0}) twice, we obtain
$$ q_{_1}''(e) = q_{_1}''(e) + q_{_2}''(e) + \cdots = 0 $$
where $q_{_1}''(e)(x) = f_{_1}(x, \cdot) +  f_{_1}( \cdot,x) -  f_{_1}(e, \cdot)-  f_{_1}(\cdot,e) \in L(V)$
for $x\in V$.
It follows that $p_{_1}(x)= f_{_1}(x,x) =0$. Differentiating repeatedly then gives
$p_{_2} = p_{_3} = \cdots = 0$ and $h=p_{_0}$ is linear.

To introduce the  concept of a symmetric Banach manifold, we begin
with the notion of a {\it symmetry} of a manifold. Let $M$ be a Banach manifold endowed with a
compatible tangent norm $\nu$ and let $p\in M$. A {\it $\nu$-symmetry}
(or {\it symmetry}, if $\nu$ is understood) at $p$ is a $\nu$-isometry
$$s: M \longrightarrow M$$ satisfying the following two conditions:
\begin{enumerate}
\item [(i)] $s$ is involutive, that is, $s^2$ is the identity map on $M$,
\item [(ii)] $p$ is an isolated fixed-point of $s$, in other words,
$p$ is the only point in some neighbourhood of $p$ satisfying $s(p)=p$.
\end{enumerate}

\begin{defin} By a {\it symmetric Banach manifold} (with a tangent norm $\nu$),
we mean a {\it connected} Banach manifold $M$, equipped with a compatible tangent norm $\nu$, such that
there is a unique $\nu$-symmetry $s_p: M \longrightarrow M$ at each $p\in M$.
\end{defin}

By definition, a {\it Finsler symmetric cone} $\Omega$ in a real Banach space $V$ is a symmetric Banach manifold
of which the tangent norm is $G(\Omega)$-invariant.

\begin{exa}
Riemannian symmetric spaces are (real) symmetric Banach manifolds
(in the Riemannian metric).  A {\it bounded symmetric domain} is  a bounded domain $D$
in  a complex Banach space such that  for each $p\in D$,
there is an involutive biholomorphic map $s_p: D \longrightarrow D$
(necessarily unique) of which $p$ is an isolated fixed-point.  Equipped with the Carath\'eodory
metric, a bounded symmetric domain is a complex symmetric Banach manifold
and  $s_p$ is the symmetry at $p$.  Finite dimensional Hermitian symmetric spaces of non-compact
type are exactly the bounded symmetric domains in $\mathbb{C}^d$ via the Harish-Chandra realisation
and have been classified by \'E. Cartan \cite{C}.

\end{exa}

\begin{exa}\label{loos} A concept of a symmetric manifold has been introduced by Loos in \cite{loos}
(see also \cite{ber}), where a
connected (real) smooth manifold $M$ is called a {\it symmetric space} if there is a smooth map
$$\mu: (x,y) \in M \times M \mapsto x\cdot y \in M$$
satisfying the following axioms:
\begin{enumerate}
\item[(i)]  $ x\cdot x= x$ ;
\item[(ii)] $x\cdot (x\cdot y) =y$;
\item[(iii)] $x\cdot(y\cdot z) = (x\cdot y)\cdot (x \cdot z)$;
\item[(iv)] there is a neighbourhood $U$ of $x$ such that $x\cdot y =y \in U$ implies $x=y$
\end{enumerate}
for all $x, y, z \in M$.
We will call $(\Omega, \mu)$ a {\it  Loos symmetric space}.
The {\it `left multiplication'} $S(x) : y\in M \mapsto x\cdot y\in M$
 is called a {\it symmetry around } $x$ in \cite{loos}.
A diffeomorphism $f: M \longrightarrow M$  is called a {\it $\mu$-automorphism}
 if $f(x\cdot y) = f(x) \cdot f(y)$.

\begin{lemma}\label{f=g}
Let $f,g: M \longrightarrow M$ be $\mu$-automorphisms on a Loss symmetric space
$(M, \mu)$  such that $f(x)=g(x)$ and $f'(x)=g'(x)$ at some point $x\in M$.
Then we have $f=g$.
\end{lemma}
\begin{proof} This follows from \cite[Lemma 3.5, Theorem 3.6]{neeb}
since $M$ is a connected manifold with spray.
\end{proof}

Given a (real) symmetric Banach manifold $M$, one can define $\mu: M \times M \longrightarrow M$ by
$$\mu(x,y) = s_x(y) \quad (s_x \mbox{ is the symmetry at } x)$$
which makes $(M,\mu)$ into a Loos symmetric space and $s_x = S(x)$.
\end{exa}

A Loos symmetric space $(M,\mu)$ is equipped with a canonical affine connection \cite[Theorem 26.3]{ber},
which is geodesically complete \cite[Theorem 3.6]{neeb}.  The derivative $$S(p)'(p) : T_pM \longrightarrow T_pM$$
of the symmetry $S(p)$ equals $-id$, where $id$ is the identity map \cite[Lemma 3.2]{neeb}.
Given a geodesic $\gamma : \mathbb{R}\longrightarrow M$
through $p$ with $\gamma(0)=p$, the symmetry $S(p)$ reverses $\gamma$ in that $S(p)(\gamma(t)) = \gamma (-t)$.

\section{Jordan algebras and order structures}\label{jordanorder}

For later applications, we review some basics of Jordan algebras, first introduced in
\cite{jvw},  and refer
to \cite{book,up} for more details. We also prove some relevant order-theoretical results in this section.
In what follows, a Jordan algebra $\mathcal{A}$ is a real vector
space, which can be infinite dimensional, equipped with a bilinear product $(a,b) \in \mathcal{A} \times \mathcal{A} \mapsto ab \in \mathcal{A}$
that is commutative, but not necessarily associative, and satisfies the {\it Jordan identity}
$$a(ba^2) = (ab)a^2 \qquad (a,b \in \mathcal{A}).$$

A vector space $\mathcal{A}$ equipped with a bilinear product will be called an {\it algebra}.
For each element $a$ in an algebra $\mathcal{A}$, we define inductively
$$a^1 =a, a^{n+1} = aa^n \qquad (n =1, 2, \ldots)$$ and call $\mathcal{A}$ {\it power associative} if
$$a^m a^n = a^{m+n} \qquad (m, n= 1,2, \ldots ).$$
We call $\mathcal{A}$ {\it unital} if it contains an identity. Evidently, if $\mathcal{A}$ is unital and power
associative, then the subalgebra  $\mathcal{J}(a,e)$ in $\mathcal{A}$ generated by $a$ and the identity $e$
is associative.

A linear map $\delta: V \longrightarrow V$ on an algebra $V$ is called a {\it derivation} if it satisfies
$$\delta(ab) = \delta(a)b + a\delta(b) \qquad (a,b\in V)$$
which can be rephrased as
\begin{equation}\label{derivation}
[\delta, L_a] = L_{\delta(a)} \qquad (a\in V)
\end{equation}
where $L_a: V \longrightarrow V$ is the {\it left multiplication}
$L_a(x)=ax$ for $x\in V$, and $[\delta, L_a] = \delta L_a - L_a \delta$
is the usual {\it commutator}. Given a derivation $\delta$ on $V$ and $a\in V$,
a simple induction shows that
\begin{equation}\label{square}
\delta (a) =0 \Rightarrow \delta(a^n)=0  \qquad (n=2,3,\ldots).
\end{equation}
Further, if $V$ is commutative, then $\delta(a^2) =0$ implies
\begin{equation}\label{square2}
 2a\delta (a) = \delta(a^2)=0.
\end{equation}

We will make use of the following result, which follows from
\cite[Lemma 2.4.4]{palacios}.
\begin{lemma}\label{244} Let $V$ be a commutative algebra on which
the commutator $[L_x, L_y]$ is a derivation for all
$x,y\in V$. Then for all $a\in V$,  we have
\begin{enumerate}
\item[\rm (i)] $[L_a, L_{a^3}] = 3L_a[L_a,L_{a^2}];$
\item[\rm (ii)] $[[L_a, L_{a^2}], [[L_a, L_{a^2}], L_{a^2}]] =0.$
\end{enumerate}
\end{lemma}
\begin{proof} (i) This is proved in \cite[Lemma 2.4.5]{palacios}.
(ii) Using (i), a simple argument in \cite[Lemma 2.4.4]{palacios} gives $[L_a, L_{a^2}]^2(a^2)=0$.
Applying (\ref{derivation})  twice yields
$$\quad [[L_a, L_{a^2}], [[L_a, L_{a^2}], L_{a^2}]] = [[L_a, L_{a^2}], L_{[L_a, L_{a^2}](a^2)}]
= L_{[L_a, L_{a^2}][L_a, L_{a^2}](a^2)}=0. \quad\qedhere$$
\end{proof}

Jordan algebras are power associative.
An element $a$ in a Jordan algebra $ \mathcal{A}$ with identity $e$ is called {\it invertible} if there exists an
element  $a^{-1}\in \mathcal{A}$ (which is necessarily unique)  such that
$aa^{-1}=e$ and $(a^2)a^{-1}=a$.
A Jordan algebra $\mathcal{A}$ is called {\it formally real} if $a_1^2 + \cdots +a_n^2 =0$ implies
$a_1 = \cdots =a_n =0$ for any $a_1, \ldots, a_n \in \mathcal{A}$ \cite{jvw}.
 A finite dimensional formally real Jordan algebra $\mathcal{A}$ is
necessarily unital (cf.\,\cite[Proposition 1.1.13]{chu}).

On a Jordan algebra $\mathcal{A}$, one can define a {\it Jordan triple product} by
$$\{a,b,c\} = (ab)c + a(bc) - b(ac) \qquad (a,b,c\in \mathcal{A})$$
which plays an important role in the structures of $\mathcal{A}$.

A real Jordan algebra $\mathcal A$ is called a {\it JB-algebra}
 if it is also a Banach space and the norm
satisfies
$$ \|a b\| \leq \|a\| \|b\|, \quad \|a^2\| =  \|a\|^2, \quad
\|a^2\| \leq \|a^2 + b^2\|$$ for all $a, b \in \mathcal A$.
A JB-algebra $\mathcal{A}$ admits a natural order structure determined by the set
$$\mathcal{A}_+ =\{x^2: x\in \mathcal{A}\}$$
which  forms a closed cone \cite[Lemma 3.3.5, Lemma 3.3.7]{stormer}
and satisfies $\mathcal{A}_+ \cap - \mathcal{A}_+ =\{0\}$. In finite dimensions,
JB-algebras are exactly the formally real Jordan algebras \cite[Lemma 2.3.7]{book}.

Let $V$ be a real Banach space.
By a {\it cone} $\Omega$ in $V$, we mean a  {\it nonempty} subset of $V$ satisfying
(i) $\Omega + \Omega \subset \Omega$ and (ii) $\alpha \Omega \subset \Omega$ for all $\alpha >0$. We note that
a cone is necessarily convex.
Trivially, $V$ itself is a cone. In the sequel, we shall exclude this case.
If  $\Omega$  is an open cone properly contained in $V$, then we must have $0\notin \Omega$ although the closure $\overline \Omega$
contains $0$.

Let $\Omega$ be an open cone properly contained in a real Banach space $V$ with norm $\|\cdot\|$, and let $\leq$ be
the partial order defined by the closure $\overline \Omega$, which is a cone,  so that
$$x\leq y \Leftrightarrow y-x \in \overline\Omega.$$
We also write $y \geq x$ for $x \leq y$.
Let $V^*$ be the dual Banach space of $V$, consisting of continuous linear functionals on $V$.
As usual,  a linear functional $f : V \longrightarrow \mathbb{R}$ is called {\it positive}
 if  $f(\overline \Omega) \subset [0, \infty)$.
By the Hahn-Banach separation theorem,  we have
$$\overline\Omega = \{v\in V: f(v) \geq 0 \mbox{ for each  $f \in V^*$
satisfying $f(\overline \Omega) \subset [0, \infty)$} \}.$$
We note that each element $e\in \Omega$ is an {\it order unit}, that is, for each $v\in V$,  we have
$$-\lambda v \leq v \leq \lambda e$$
for some $\lambda >0$. Indeed, since $\Omega$ is open, $ e-\Omega$  is a neighbourhood of $0\in V$
and therefore one can find $\lambda >0$ such that $\pm \lambda v \in e- \Omega$, which gives
$ \lambda v = e - a_1 $  and $ -\lambda v = e-a_2$ for some $a_1, a_2 \in \Omega$. In other words,
$$ - \frac{1}{\lambda} e \leq v \leq \frac{1}{\lambda} e.$$
The preceding argument also implies
\begin{equation}\label{gen}
V= \Omega -\Omega.
\end{equation}
 An  order unit $e\in \Omega$ induces a semi-norm $\|\cdot\|_e$ on $V$,
 defined by
 \[ \|x\|_e = \inf\{\lambda>0: -\lambda e \leq x \leq \lambda e\} \qquad (x\in V)\]
which satisfies
\begin{equation}\label{eq}
-\|x\|_e e\leq x \leq \|x\|_e e
\end{equation}  and
\begin{equation}\label{ball}
\{x\in V: \|x\|_e \leq 1\} = \{x\in V: -e \leq x\leq e\}.
\end{equation}
Since $\{x\in V: \|x\|_e =0\} = \overline \Omega \cap -\overline\Omega$, the semi-norm $\|\cdot\|_e$ is a norm
if and only if  $$ \overline \Omega \cap -\overline\Omega=\{0\}$$
 in which case $\Omega$ is called a {\it proper cone}
and $\|\cdot\|_e$ is
called the {\it order-unit norm} induced by $e$. All order-unit norms induced by elements in $\Omega$ are mutually equivalent.

Henceforth, let $\Omega$ be a proper open cone in $V$.
It follows from (\ref{ball})  that every linear map $\psi: V \longrightarrow V$ which is {\it positive}, meaning $\psi(\overline \Omega)
\subset \overline \Omega$, is continuous with respect to the order-unit norm $\|\cdot\|_e$ and moreover,
$\|\psi\|_e = \|\psi(e)\|_e$, where the former denotes the norm of $\psi$ with respect to $\|\cdot\|_e$.
In particular, if $\psi:V \longrightarrow \mathbb{R}$ is a positive linear functional, then $\|\psi\|_e = \psi(e)$.

Let  $(V, \|\cdot\|_e)$  denote the vector space $V$ equipped with the order-unit norm $\|\cdot\|_e$, and
$(V, \|\cdot\|_e)^*$ its dual space. A positive linear map $\psi : (V, \|\cdot\|_e) \longrightarrow (V, \|\cdot\|_e)$
is an isometry
if and only if $\psi(e)=e$ \cite[Proposition 2.3]{chu}.
By \cite[Lemma 2.5]{chu}, there is a positive constant $c>0$ such that
\begin{equation}\label{c}
\|\cdot\|_e \leq c \|\cdot\|.
\end{equation} It follows that every $\|\cdot\|_e$-continuous linear functional on $V$
is also $\|\cdot\|$-continuous. On the other hand, given $f\in V^*$ satisfying $f(e)=1=\|f\|_e$,
then $f$ is positive and hence continuous with
respect to the norm $\|\cdot\|_e$.

Denote the {\it state space} (with respect to the order unit $e$) by
\begin{equation}\label{s}
S_e= \{f\in (V^, \|\cdot\|_e)^*: f(e)= 1=\|f\|_e\} =  \{f\in V^*: f(e)= 1, f \mbox{ is positive} \}
\end{equation}
which is a weak* compact convex set in the dual $V^*$ and  we have
$$\|x\|_e = \sup \{|f(v)|: f \in S_e\} \qquad (x\in V)$$
 (cf.\,\cite[Lemma 1.2.5]{stormer}).

\begin{lemma}\label{1}
Let $\Omega$ be a proper open cone in a real Banach space $V$ and let $e\in \Omega$,
which induces an order-unit norm $\|\cdot\|_e$ on $V$.
Then we have
$$\Omega = \bigcap_{f\in S_e} f^{-1}(0, \infty).$$
\end{lemma}

\begin{proof}
Given that $V$ is partially ordered by the closure $\overline \Omega$,  we have
\begin{equation}\label{po}
\overline \Omega = \bigcap_{f\in S_e} f^{-1}[0, \infty)
\end{equation}
since $f/f(e) \in S_e$ for each nonzero positive linear functional $f \in V^*$.

Let $a\in \Omega$. Then for each $f\in S_e$, we have
 $f(a)>0$ since $a$ is an order unit, which implies $ e \leq \lambda a$ for some constant $\lambda >0$ and hence
$1\leq \lambda f(a)$. This proves
$$ \Omega \subset  \bigcap_{f\in S_e} f^{-1}(0, \infty).$$

Conversely, let $a\in V$ and  $f(a)>0$ for all $f\in S_e$. Then $a \in \overline\Omega$ and  by weak* compactness of $S_e$,
one can find some $\delta >0$ such that $f(a) \geq \delta$ for all  $f\in S_e$. Let
 $$N=\left\{x\in V: \|x-a\| < \frac{\delta}{2c}\right\} \subset \left\{x\in V: \|x-a\|_e < \frac{\delta}{2}\right\} $$
where $c>0$ is given in (\ref{c}). Then $N$ is an open neighbourhood of $a$ and, $N \subset \overline \Omega$ since
$$ x\in N \Rightarrow -\frac{\delta}{2} e\leq x-a \Rightarrow a- \frac{\delta}{2}e \leq x \Rightarrow \frac{\delta}{2}\leq f(x)$$
for all $f\in S_e$. Hence $a$ belongs to the interior  $\overline \Omega^0$ of
$\overline\Omega$ and, as $\Omega$ is open and convex, we have $\Omega= \overline \Omega^0$ and $a\in \Omega$.
\end{proof}

We see from (\ref{c}) that if $\dim V<\infty$, then the order-unit norm $\|\cdot\|_e$ is equivalent to
the norm of $V$ by the open mapping theorem. In fact, the equivalence of the two norms is related to
the basic concept of a normal cone in the theory of partially ordered topological vector spaces.

\begin{lemma}\label{norm} Let $\Omega$ be a proper open cone in a real Banach space $V$ with norm $\|\cdot\|$.
Then the order-unit norm $\|\cdot\|_e$ induced by $e\in \Omega$ is equivalent to $\|\cdot\|$ if and only if
$\Omega$ is a normal cone in $V$, that is, there is a constant $\gamma >0$ such that $0\leq x \leq y$
implies $\|x\|\leq \gamma\|y\|$ for all $x,y \in V$.
In particular, $(V, \|\cdot\|_e)$ is a Banach space if $\Omega$ is a normal cone.
\end{lemma}

\begin{proof} By the definition of the order-unit norm $\|\cdot\|_e$, we have $0\leq x\leq y$ in $V$
implies $\|x\|_e \leq \|y\|_e$. Hence $\Omega$ is normal in $(V,\|\cdot\|_e)$. If $\|\cdot\|$ is equivalent to $\|\cdot\|_e$,
then evidently $\Omega$ is also normal in $(V,\|\cdot\|)$.

Conversely, let $\Omega$ be normal in $(V, \|\cdot\|)$.
We have already noted in (\ref{c}) that $\|\cdot\|_e \leq c \|\cdot\|$ for some constant $c>0$.
 By (\ref{ball}) and normality of $\Omega$,
there is a constant $\gamma>0$ such that
$$\|x\|_e \leq 1 \Leftrightarrow -e\leq x\leq e \Rightarrow 0\leq x+e\leq 2e \Rightarrow \|x+e\| \leq 2\gamma\|e\|
\Rightarrow \|x\| < 2(\gamma +1)\|e\|$$
which implies $\|\cdot\| \leq 2(\gamma +1)\|e\| \|\cdot\|_e$ and the equivalence of $\|\cdot\|$ and $\|\cdot\|_e$.
\end{proof}

We note that a self-dual cone $\Omega$ in a Hilbert space $H$ is a proper cone,
and also normal since it has been shown in
\cite[Lemma 2.6]{chu} that the order-unit norms induced by elements in $\Omega$ are all equivalent to the norm
of $H$.

Let  $L(W)$ be  the Banach algebra
of bounded linear operators on a complex Banach space $W$ and $I\in L(W)$ the identity operator.
We recall that an element $T\in L(W)$ is called {\it hermitian} if its numerical range
$\textsf{V}(T)$ is contained in $\mathbb{R}$, where
$$\textsf{V}(T) =\{\psi(T): \psi \in L(W)^* \mbox{ satisfies } \|\psi\|=1=\psi(I)\},$$
which is equivalent to
$$\|\exp itT\| = \|I + itT + (itT)^2/2! + \cdots \| =1 \qquad (t\in \mathbb{R})$$
(cf.\,\cite[Chapter 2]{bd}).  If $T_0\in L(W)$ is hermitian, then the left multiplication
$$L_{T_0} :S\in L(W) \mapsto T_0S\in L(W)$$ is a hermitian operator in L(L(W)) because
the linear map $T\in L(W) \mapsto L_T\in L(L(W))$ is an isometry.

\begin{lemma}\label{95} Let $\eta : L(W) \longrightarrow L(W)$ be a hermitian operator.
Then for all $T\in L(W)$, we have $\|\eta(T)\|^2 \leq 4 \|T\|\|\eta^2(T)\|$.
\end{lemma}
\begin{proof} This is proved in \cite[p.\,95]{bd}.
\end{proof}

Given a real Banach space  $V$,
 one can equip its complexification
$V_c = V \otimes \mathbb{C}=V\oplus iV$ with a norm $\|\cdot\|_c$ so that $(V_c,
\|\cdot\|_c)$ is a complex Banach space and
\begin{enumerate}
\item[\rm(i)] the isometric
embedding $v \in V \mapsto (v,0)\in V\oplus iV$ identifies $V$ as
a real closed subspace of $V_c$;
\item [\rm(ii)] the map $T\in L(V) \mapsto T_c\in L(V_c)$ is isometric,
where $T_c$ is the complexification of $T$ defined by
 $T_c(x+iy) = T(x) + i T(y)$ for $x,y \in V$.
\end{enumerate}
Moreover, if $V$ is an algebra
 satisfying $\|xy\|\leq \|x\|\|y\|$ for all $x,y \in V$, the norm $\|\cdot\|_c$ can be chosen
so that $\|ab\|_c\leq \|a\|_c\|b\|_c$ for all $a,b \in V_c$. In this case, the linear map
\begin{equation}\label{l_a}
a\in V_c\mapsto L_a \in L(V_c)
\end{equation}
is an isometry, where $L_a$ is
the left multiplication. In the sequel, we will make use of this
construction.

In the preceding construction, if the norm of $V$  is an order-unit norm $\|\cdot\|_e$, one can also define
a notion of {\it numerical range} $\textsf{v}(a)$ of an element $a\in V_c$ by
$$\textsf{v}(a)= \{f(a): f\in V_c^* \mbox{ satisfies } \|f\|=1=f(e)\}.$$
If  $V$ is an algebra and the order unit $e$ is an algebra  identity, then an application of  the isometry in (\ref{l_a}) implies
$\textsf{V}(L_a) \subset \textsf{v}(a)$ and therefore $L_a$ is hermitian if $\textsf{v}(a) \subset \mathbb{R}$.

\section{Tube domains over Finsler symmetric cones}\label{mains}

We prove the main theorem  in this section.  Let $\Omega$ be a proper open cone
 in a real Banach space $(V, \|\cdot\|)$. Then it is a real connected Banach manifold modelled on $V$.
Let $(V_c, \|\cdot\|_c)$  be a complexification of $V$.
The domain
$$ D(\Omega):=V \oplus i\Omega = \{v+ i\omega: v\in V, \omega \in \Omega\}\subset V_c = V\oplus iV$$
in $V_c$ is called a {\it tube domain} over $\Omega$.

Let $V \oplus i\Omega$ be biholomorphic to
a bounded domain (this is always the case if $\dim V <\infty$ \cite[Chapter II, Sec.\,5]{kob}).
On $D(\Omega)=V \oplus i\Omega$, the Carath\'eodory distance $\rho$ is defined, in terms of the
Poincar\'e distance $\rho_\mathbb{D}$ on $\mathbb{D}$, by
$$\rho(z,w) := \sup\{ \rho_{\mathbb{D}}(f(z), f(w)): f\in H(D(\Omega), \mathbb{D})\} \qquad(z,w\in D(\Omega) )$$
which need not coincide with the integrated distance of  the
Carath\'eodory differential metric $\mathcal{C}$ on $V \oplus i\Omega$,  defined in Example \ref{cara}.

If the proper open cone $\Omega$ in $V$ is  normal,
then the order-unit norms induced by elements in $\Omega$ are all equivalent to $\|\cdot\|$ by Lemma \ref{norm}
and one can define a compatible tangent norm  $\tau$ on $\Omega$ by
\begin{equation}\label{tau}
\tau(p,v) = \|v\|_p \qquad ((p,v) \in \Omega \times V)
\end{equation}
where $\|\cdot\|_p$ denotes the order-unit norm induced by the order unit $p\in \Omega$. To see that $\tau$ is continuous,
let $(p_n)$ converge to $p$ in $\Omega$ and $(v_n)$ converge to $v$ in $V$. Given $1>\ep >0$,
 $\|p_n-p\|_p \rightarrow 0$ implies $-\ep p \leq p_n-p\leq \ep p$ and $(1-\ep )p \leq p_n \leq (1+\ep)p$ from some $n$ onwards,
 which gives
 $$-(1+\ep)\|v_n\|_{p_n} p \leq - \|v_n\|_{p_n}p_n \leq v_n \leq \|v_n\|_{p_n}p_n \leq (1+\ep) \|v_n\|_{p_n}p$$
 and hence $\|v_n\|_p \leq (1+\ep)\|v_n\|_{p_n}$. Likewise $p \leq \frac{p_n}{1-\ep}$ implies $\|v_n\|_{p_n} \leq \frac{\|v_n\|_p}{1-\ep}$
 and therefore
 $$1-\ep \leq \frac{\|v_n\|_p}{\|v_n\|_{p_n}} \leq 1+ \ep.$$
 Since $\|v_n\|_p \rightarrow \|v\|_p$ as $n \rightarrow \infty$,
 we conclude $\|v_n\|_{p_n} \rightarrow \|v\|_p$, proving continuity of $\tau$.
 The above argument also implies that for each   $a\in \mathcal{U}:= \{v\in V:\|v-p\|_p < \varepsilon <1\}$,
 we have $\|v\|_p/(1+\varepsilon) \leq \|v\|_a \leq \|v\|_p/(1-\varepsilon)$ for all $v\in V$.
 Hence $\tau$ is a compatible tangent norm.

The tangent norm $\tau$ coincides with the tangent norm $b: T\Omega \longrightarrow [0,\infty)$  in \cite[12.31, 22.37]{up}, which is
defined as follows. Fix $e\in \Omega$. Then each $g\in G(\Omega)$ satisfying $g(e)=e$ is an isometry with respect to the
order unit norm $\|\cdot\|_e$  and hence one can define
\begin{equation}\label{b}
b(p,v) = \|h(v)\|_e \qquad ( (p,v) \in T\Omega)
\end{equation}
for any $h\in G(\Omega)$ satisfying $h(p)=e$. In fact, $\tau$ is $G(\Omega)$-invariant, which implies $\tau=b$. For if
$h\in G(\Omega)$, then we have
$$\tau(h(p), h'(p)(v)) = \tau(h(p), h(v)) = \|h(v)\|_{h(p)} = \|v\|_p = \tau(v,p) \qquad(v\in T_p\Omega =V)$$
where the third identity follows from the equivalent conditions
$$-\lambda h(p) \leq h(v) \leq \lambda h(p) \Leftrightarrow\lambda p \leq v \leq \lambda p \qquad (\lambda >0).$$

 By \cite[Lemma 1.3, Theorem 1.1]{nus}, the integrated distance $d_\tau$ of $\tau$ on $\Omega$ coincides with
 Thompson's metric
\[ d_\tau (x,y) =\max\{ \log M(x/y),\, \log M(y/x)\} \qquad (x,y \in \Omega)\]
where
\[M(a/b):= \inf\{\beta>0:  \beta a \geq b \} \qquad (a, b\in \Omega).
\]
It has been shown in \cite[(5.3);\,Theorem II]{v} that the restriction of the Carath\'eodory distance $\rho$
to $i\Omega$ can be expressed as
$$\rho(ix,iy) = \sup \left\{\log \left|\frac{f(x)}{f(y}\right|: f\in V^*, f(\Omega) \subset (0,\infty)\right\} \quad (x, y \in \Omega).$$
From this one can deduce that $d_\tau(x,y) = \rho(ix,iy)$, as shown in \cite[Lemma 3.6.17]{chu1}.

\begin{exa}\label{exa} Let $\mathcal{A}$ be a JB-algebra with identity $e$, partially ordered by the
closed cone $\mathcal{A}_+ = \{a^2:a\in \mathcal{A}\}$. Let $\Omega$ be the interior of  $\mathcal{A}_+$.
 Then $e\in \Omega$ is an order unit and
the norm of $\mathcal{A}$ coincides with the order-unit norm $\|\cdot\|_e$.
Hence $\Omega$ is a normal cone. Equip $\Omega$  with the tangent norm $\tau$ defined in (\ref{tau}).
Each element $a\in \Omega$ is
invertible  and one can define a smooth map
$\mu: \Omega \times \Omega\longrightarrow \Omega$ in terms of the Jordan triple product
by $$\mu(x,y) = \{x, y^{-1}, x\} \qquad (x,y \in \Omega).$$
It can be shown that $(\Omega, \mu)$ is a Loos symmetric space (e.g.\,\cite{ll}) and moreover,
each $\tau$-isometry is a $\mu$-homomorphism.
By Lemma \ref{f=g}, a $\tau$-symmetry
$s_p: \Omega \longrightarrow \Omega$ at $p\in \Omega$
must be unique since
$s_p'(p) = - id : T_p\Omega \longrightarrow T_p \Omega$.
\end{exa}

Finally, we are  ready to prove the main result.

\begin{theorem}\label{main} Let $\Omega$ be a proper open cone in a real Banach space $V$, with closure $\overline\Omega$.
The following conditions are equivalent.
\begin{enumerate}[\upshape (i)]
\item The Siegel domain $V \oplus i\Omega$ is biholomorphic to a bounded symmetric domain.
\item $\Omega$ is a normal linearly homogeneous Finsler symmetric cone.
\item $V$ is a unital JB-algebra in an equivalent norm and $\overline\Omega = \{a^2: a\in V\}.$
\end{enumerate}
\end{theorem}

\begin{proof} (i) $\Leftrightarrow$ (iii). This has been proved in \cite{bku, ku}.

(iii) $\Rightarrow$ (ii). This is essentially proved in \cite{bku,ku},  more
details can be found in \cite[22.37]{up}.
It suffices to highlight the main arguments.
 First, $\Omega$ is a normal cone as noted in Example \ref{exa}.
Let $e\in V$ be the algebra identity. Then  $e\in \Omega$ and each element
in $\Omega$ is invertible. The linear automorphism group
$G(\Omega)$ acts transitively on $\Omega$ and the tangent norm  $b:T\Omega \longrightarrow [0,\infty)$
defined in (\ref{b}) is $G(\Omega)$-invariant. Equipped with  this tangent
norm,  the inverse map $x\in \Omega \mapsto x^{-1}\in \Omega$ is a $b$-symmetry at $e$, which is
unique, as noted in Example \ref{exa}, and  hence $\Omega$ is a symmetric Banach manifold
by linear homogeneity.

(ii) $\Rightarrow$ (iiii).  Let $\Omega$ be a normal linearly homogeneous Finsler symmetric cone in a compatible
$G(\Omega)$-invariant tangent norm $\nu$.
 For each $p\in \Omega$, let
$s_p : \Omega \longrightarrow \Omega$ be the symmetry at $p$. By Example \ref{loos}, $(\Omega, \mu)$ is a Loos
symmetric space, with the smooth map
$$\mu: (x,y) \in \Omega \times \Omega \mapsto x \cdot y = s_x(y) \in \Omega.$$

Denote by Diff$(\Omega)$ the diffeomorphism group
of $\Omega$ and let
$${\rm Aut}\,\Omega = \{f\in {\rm Diff}(\Omega): f\circ s_p = s_{f(p)}\circ f, \forall p\in \Omega\}$$
be the subgroup of Diff$(\Omega)$, consisting of  $\mu$-automorphisms of $\Omega$.

By \cite[Theorem 2.4, Theorem 5.12]{klotz},
Aut\,$\Omega$ carries the structure of a real Banach Lie group, with Lie algebra
\begin{equation}\label{kill}
{\rm Kill}\,\Omega =\{X\in\mathcal{V}(\Omega): \exp tX \in {\rm Aut}\,\Omega, \forall t \in \mathbb{R}\}
\end{equation}
which is a Banach Lie algebra in some norm $|\cdot|$ and a subalgebra of the
Lie algebra $\mathcal{V}(\Omega)$ of smooth vector fields on $\Omega$.

We note that
 the linear automorphism group $G(\Omega)$  is contained in Aut\,$\Omega$.
Indeed, given $p\in \Omega$ and $g \in G(\Omega)$, the composite map
$$g^{-1}\circ s_{g(p)}\circ g : \Omega \longrightarrow \Omega$$ is
 a $\nu$-isometry by $G(\Omega)$-invariance of $\nu$, with isolated fixed-point $p$. Hence by uniqueness of the
symmetry $s_p$, we have  $g^{-1}\circ s_{g(p)}\circ g = s_p$ and $g \in {\rm Aut}\, \Omega$.
 It follows that $\frak g(\Omega) \subset
{\rm Kill}\,\Omega$ by (\ref{go}) and (\ref{kill}).

Fix a point $e\in \Omega$, which induces an order-unit norm $\|\cdot\|_e$ on $V$, equivalent to the
norm $\|\cdot\|$ of $V$, by Lemma \ref{norm}.

The  evaluation map
$$X\in {\rm Kill}\,\Omega \mapsto X(e) \in V$$
is surjective by \cite[Proposition 5.9]{ber} (cf.\,\cite[Theorem II.2.2]{loos}).
In fact, the differential of the orbital map $\rho: g\in G(\Omega) \mapsto g(e) \in \Omega$
 at the identity of $G(\Omega)$ is the evaluation map
\begin{equation}\label{onto}
X\in \frak g(\Omega) \mapsto X(e) \in T_e \Omega =V
\end{equation}
which is also surjective by linear homogeneity of $\Omega$ \cite{chu} (cf.\,\cite[p.\,110]{w}).

Let $s_e: \Omega \longrightarrow \Omega$ be the symmetry at $e$. Then $s_e \in {\rm Aut}\, \Omega$.
Since $s_e^2$ is the identity map, the adjoint representation
$$\theta = Ad(s_e): {\rm Kill}\,\Omega \longrightarrow {\rm Kill}\,\Omega$$
is an involution and the Lie algebra ${\rm Kill}\,\Omega$ has an eigenspace decomposition
\begin{equation*}
{\rm Kill}\,\Omega = \frak k \oplus \frak p
\end{equation*}
with
\begin{equation}\label{[kp]}
[\frak k, \frak k] \subset \frak k, \quad [\frak k, \frak p] \subset \frak p, \quad [\frak p, \frak p] \subset \frak k
\end{equation}
where $\frak k$ is the $1$-eigenspace and $\frak p$ the $(-1)$-eigenspace, both are
$|\cdot|$-closed. Moreover, we have as usual (e.g. \cite[Lemma 2.4.5]{book})
$$\frak k=\{X\in {\rm Kill}\,\Omega: X(e)=0\} = \{X\in {\rm Kill}\,\Omega:  \exp tX (e) =e, \forall t\in \mathbb{R}\}.$$
Hence the linear map
\begin{equation}\label{map}
X\in \frak p \mapsto X(e) \in V
\end{equation}
 is bijective as $\frak k \cap \frak p=\{0\}$.

Let  $I\in L(V)$ be the identity vector field, which belongs to ${\rm Kill}\, \Omega$ since
 $\exp t I = \varepsilon^t I \in G(\Omega)$ for all $t\in \mathbb{R}$, where $\varepsilon = \log^{-1}(1)$
 denotes Euler's number, to avoid confusion with the order unit $e\in \Omega$.
Hence  $ [I,X] \in {\rm Kill}\, \Omega$ for all $X \in {\rm Kill}\,\Omega$.
We show  $\theta I = -I$.

We have
 $$(\theta I)(\cdot) =  \displaystyle \left.\frac{d}{dt} \right|_{t=0}
   \exp t\theta I (\cdot)=  \displaystyle \left.\frac{d}{dt} \right|_{t=0} s_e (\exp tI)s_e (\cdot) =
   \displaystyle \left.\frac{d}{dt} \right|_{t=0} s_e (\varepsilon^ts_e )(\cdot).$$
Since the symmetry $s_e$ reverses the geodesic $\gamma(t) = \exp tI(e) = \varepsilon^te$,
we have $s_e(\varepsilon^te) = s_e(\gamma(t)) = \gamma(-t) = \varepsilon^{-t}e$. By uniqueness of the symmetry,
we have $\varepsilon^{t}s_e(\varepsilon^{t}\cdot) = s_e(\cdot)$, which gives
 $$(\theta I)(\cdot) =  \displaystyle \left.\frac{d}{dt} \right|_{t=0}\varepsilon^{-t}I(\cdot)= -I(\cdot).$$

We show next that each $X= f \frac{\partial}{\partial x} \in \frak p$ is a linear vector field. For this, we first note that
  $X = Z - \theta Z$ for some $Z\in \frak g(\Omega) \subset {\rm Kill}\, \Omega$. Indeed, (\ref{onto}) implies
  the existence of $Y\in \frak g(\Omega)$ such that $Y(e) = X(e)$, which gives
$$X(e)= Y(e)= \frac{1}{2}(Y + \theta Y)(e) +  \frac{1}{2}(Y - \theta Y)(e) =  \frac{1}{2}(Y - \theta Y)(e)$$
since $ \frac{1}{2}(Y + \theta Y) \in\frak k$. It follows that $X=  \frac{1}{2}(Y - \theta Y) \in \frak p$, where
$Z= \frac{1}{2} Y\in \frak g(\Omega)$. Since $Z$ is a linear vector field by (\ref{go}), linearity of
$X=Z -\theta Z$ follows from that of $\theta Z$. By the remarks in Section \ref{fins},
the latter is linear because
$$[I, \theta Z] = \theta [\theta I, Z] =  -\theta [I,Z]=0.$$

The linear isomorphism $X\in \frak p \mapsto X(e) \in V$ in (\ref{map})
 is a continuous bijection and hence by the open mapping
theorem, its inverse is also continuous and there is a constant $\kappa>0$ such that
 $$\kappa\|X(e)\| \geq |X|$$ for all $X\in \frak p$.
 Let $L: V \longrightarrow \frak p$ be the inverse of the map in (\ref{map}) so that
$$L(x)(e) = x \qquad (x \in V)$$
and $|L(a)| \leq \kappa\|a\|$ for all $a\in V$.

On $V$, we can now define a product
\begin{equation}\label{jproduct}
xy := L(x)(y) \qquad (x,y \in V)
\end{equation}
where $L(x)$ is a  linear vector field, identified as an element of $L(V)$.

We show that $V$ is a Jordan algebra in this product, with identity $e$.
First, we have
$$ae = L(a)(e) = a \qquad (a\in V).$$
Given $a,b \in V$, we have
$$ ab-ba =  [L(a), L(b)](e) =0$$
where $L(a), L(b) \in \frak p$ implies $[L(a), L(b)] \in \frak k$, by (\ref{[kp]}).

Before deriving the Jordan identity, we need to establish some facts.
By continuity of the evaluation map in (\ref{map}), there is a constant $\rho >0$ such that
 $\|Xe\| \leq \rho|X|$ for all $X\in \frak p$. This implies  $\|a\| = \|L(a)e\| \leq \rho|L(a)|$ and
$$
\|ab\|= \|L(a)L(b)e\| \leq \kappa\|a\|\|L(b)e\| \leq \rho \kappa^2\|a\|\|b\| \qquad (a,b \in V)
$$
as well as
\begin{equation}\label{banach}
\|ab\|_e \leq \alpha \|a\|_e\|b\|_e \qquad (a,b \in V)
\end{equation}
for some $\alpha >0$, since $\|\cdot\|$ and $\|\cdot\|_e$ are equivalent.

We begin by showing that $V$ is power associative. One can verify
directly the identity
$$[[L(x), L(y)], L(z)]( e) = L([L(x),L(y)]z)(e) \qquad (x,y,z \in V)$$
where $[L(x), L(y)] \in \frak k$ implies $[L(x), L(y)](e)=0$.
 It follows that
\begin{equation}\label{30}
[[L(x), L(y)], L(z)]  = L([L(x),L(y)]z)
\end{equation}
since both vector fields belong to $\frak p$. By definition, $L(x)$ is
the left multiplication by $x$ on the commutative algebra $V$.
By (\ref{30}) and  (\ref{derivation}), $[L(x), L(y)]$
is a derivation on  $V$ for all $x,y \in V$. Hence Lemma \ref{244} implies
\begin{equation}\label{powera}
[[L(x), L(x^2)], [[L(x), L(x^2)], L(x^2)]] =0 \qquad (x\in V).
\end{equation}

Let $a\in V$ and consider the linear vector field $T=[L(a), L(a^2)] \in \frak k$,
identified as an element of $L(V)$. Since
$\exp tT : \Omega \longrightarrow \Omega$ satisfies $\exp tT(e)=e$ for all $t\in \mathbb{R}$,
each $\exp tT$ is a positive linear map on $(V, \|\cdot\|_e)$ and $\|\exp tT\|= \|\exp tT(e)\|  =\|e\|=1$.
Let $T_c \in L(V_c)$ be the complexification of $T\in L(V)$, as defined in Section \ref{jordanorder}.
 Then we have
$$\|\exp tT_c\| = \|(\exp tT)_c\| = \|\exp tT\|=1 \qquad (t\in \mathbb{R}).$$
Hence $iT_c$ is a hermitian operator in $L(V_c)$ and it follows from (\ref{powera}) that
$$[ iT_c, [iT_c, L(a^2)_c]] = - [T,[T, L(a^2)]]_c =0.$$
The linear operator
\begin{equation}\label{eta}
\eta: S\in L(V_c) \mapsto [iT_c, S] = iT_cS - S(iT_c) \in L(V_c)
\end{equation}
is hermitian, since both the left  multiplication $S\in L(V_c) \mapsto iT_c S \in L(V_c)$
 and right multiplication $S\in L(V_c)  \mapsto S(iT_c) \in L(V_c)$  are hermitian.
Hence Lemma \ref{95}  implies
\begin{eqnarray*}
\|[iT_c, L(a^2)_c]\|^2 &=& \|\eta(L(a^2)_c)\|^2  \leq  4\|L(a^2)_c\|\|\eta^2(L(a^2)_c)\|\\
&=& 4\|L(a^2)_c\|\|[iT_c, [iT_c, L(a^2)_c]]\|=0
\end{eqnarray*}
which gives
\begin{equation}\label{234}
[[L(a), L(a^2)], L(a^2)] = [T, L(a^2)]=0.
\end{equation}
In particular, we have $$[L(a), L(a^2)](a^2) = [[L(a), L(a^2)], L(a^2)](e) =0$$
 since  $[L(a), L(a^2)](e) =0$.
Further, by Lemma \ref{244}, we have $$ L(a) T=L(a)[L(a), L(a^2)]
 =\frac{1}{3} [L(a), L(a^3)] \in\frak k$$
and hence
$TL(a) = L(a)T - [L(a), T] = L(a)T - [L(a), [L(a), L(a^2)]] \in {\rm Kill}\,\Omega$,
where $TL(a)$ is a linear vector field, identified as an element of $L(V)$.
By (\ref{square2}), we have $L(a)TL(a)(e) = a[L(a), L(a^2)](a)= 0$ and hence
$(TL(a))^2(e) = TL(a)TL(a)(e) =0$ as well as
$$(TL(a))^{n+2}(e) = (TL(a))^n(TL(a))^2(e)=0
\qquad (n=1,2, \ldots).$$ It follows that
\begin{eqnarray*}
&& \exp tTL(a) (e) = e+ tTL(a)(e) + t^2(TL(a))^2 (e)/2! + \cdots = e+tTL(a)(e)
 \end{eqnarray*}
for all $t\in \mathbb{R}$,
where $\exp tTL(a) \in {\rm Aut}\,\Omega$ implies  $e\pm tTL(a)(e) \in \Omega$
for all $t>0$. In other words,
$$- \frac{1}{t} e \leq TL(a)(e) \leq \frac{1}{t} e \qquad (t>0)$$
and therefore
$[L(a), L(a^2)](a) = TL(a)(e) = 0$. By (\ref{square}), we have
$$[L(a), L(a^2)](a^n) =0 \qquad (n=1,2,\ldots).$$
That is,  $a^{n+3}= a^{n+1} a^2  $ for $n =1,2,\ldots$. It follows that
$$[L(a), L(a^m)](a) = a^{m+2} - a^ma^2 = 0 \qquad (m =2, 3, \ldots)$$
and again,  (\ref{square})  implies
$$[L(a), L(a^m)](a^n) =0 \qquad (n, m-1 = 1,2,\ldots)$$
which gives $a^m a^{n+1} = a(a^ma^n)$ for $m,n = 1,2,\ldots$. From this we deduce
$$a^ma^n = a^{m+n} \qquad (m,n=1,2,\ldots)$$
by induction, since  $a^ma^n = a^{m+n}$ implies
 $$a^ma^{n+1} = a(a^ma^n) = aa^{m+n} = a^{m+n+1}.$$
This  proves  power associativity of $V$
and therefore the closed subalgebra $J(a,e)$ of $V$ generated by $e$ and any $a\in V$
is associative.

Since $\Omega$ is geodesically complete and the orbits of the one-parameter groups
$t\in \mathbb{R}\mapsto \exp tX~(X\in \frak p)$ are the geodesics through $e\in \Omega$
(cf.\,\cite[Example 3.9]{neeb}), we must have
$$\Omega = \{ \exp X(e): X\in \frak p\}.$$

It follows that each
$a\in \Omega$ can be written as
$a = \exp X(e)$ for some $X\in \frak p$, where $X$ is a linear vector field,
 identified as an element of $L(V)$. For each $z\in V$, define
 $$Exp\, z  =  e+ z + \frac{z^2}{2!} + \cdots. $$ Then we have
$a=\exp X(e) = e + X(e) + X^2(e)/2! + \cdots  = Exp\, x$, where $x=X(e)\in V$.
 By power associativity, we have
$a = (Exp\, \frac{x}{2})^2$. This proves the first part of the following inclusions
\begin{equation}\label{equal}
\Omega \subset \{x^2: x\in V\} \subset \overline\Omega
\end{equation}

To prove the second inclusion in (\ref{equal}), let $v \in V$. We show $v^2\in \overline\Omega$.
By a remark before (\ref{gen}), there is some $\lambda_0 >0$
 and $a_0\in \Omega$ such that $\lambda_0 v = e-a_0\in  J(a_0,e)$, where
$J(a_0,e)$ is a commutative real Banach algebra in the order-unit norm by (\ref{banach})
(cf.\,\cite{in}).

For each $x \in \Omega \cap J(a_0,e)$, we show $a_0x\in \Omega$. Indeed, given $a_0 = Exp\, z
= \exp Z(e)$ for some $z= Z(e)$ and $Z\in \frak p$,
we have $x\in J(a_0,e) \subset J(z,e)$ and associativity
of $J(z,e)$ implies
\begin{eqnarray*}
a_0x &=& x + z x + \frac{z^2 x}{2!} + \cdots =  x + z x + \frac{z(z x)}{2!} + \cdots \\
&=& x + Z(x) + \frac{Z^2(x)}{2!} + \cdots = \exp Z (x) \in \Omega.
\end{eqnarray*}

Further, for $y\in  \overline{\Omega} \cap J(a_0,e)$,
we show $a_0y\in \overline\Omega$.
The cone   $\Omega \cap J(a_0,e)$ is open in $J(a_0,e)$ and as before,
 we have
\begin{equation}\label{+-+}
J(a_0,e)= \Omega \cap J(a_0,e) - \Omega \cap J(a_0,e)
\end{equation}
and $e \in \Omega \cap J(a_0,e)$  is an order-unit in the  induced ordering of $J(a_0,e)$ with respect to the
the cone $ \overline{\Omega} \cap J(a_0,e)$. Repeating the remark before (\ref{gen}) for the cone
 $\Omega \cap J(a_0,e)$, one can find $\lambda >0$ and $w\in  \Omega \cap J(a_0,e)$ such that
$\lambda y = e-w$, where $w = e-\lambda y \leq e$ and $0< f(w) \leq 1$ for all states  $f$
in the state space $S_e$ defined in (\ref{s}).
The latter implies
$$f\left(e-\left(1-\frac{1}{n}\right)w\right) = 1- \left(1-\frac{1}{n}\right)f(w) >0 \qquad (n=1,2, \ldots)$$
for all $f\in S_e$ and hence $e-(1-1/n)w \in  \Omega \cap J(a_0,e)$ by Lemma \ref{1}. Therefore the preceding
argument yields $a_0(e-(1-1/n)w)
\in  \Omega \cap J(a_0,e)$ and
$$\lambda a_0y= \lim_n a_0(e-(1-1/n)w) \in  \overline{\Omega} \cap J(a_0,e).$$

Let $S_{a_0} = \{\psi\in J(a_0,e)^*: \psi(e)=1, \psi \mbox{ is positive on } J(a_0,e)\}$ be
the state space of $J(a_0,e)$.
Let $\psi\in S_{a_0}$ be a pure state, that is, $\psi$ is an extreme point of $S_{a_0}$.
We show that $\psi(a_0^2) = \psi (a_0)^2$.
Let $b = a_0/2\|a_0\|_e\in \Omega \cap J(a_0,e)$ so that
$\|b\|_e <1$. Then we have $0< \vp(b) <1$ for all $\vp \in S_{a_0}$ and $e-b\in \Omega \cap J(a_0,e)$ by  Lemma \ref{1}.
One can define two states
$\psi_b$ and $\psi_{e-b}$ in $S_a$ by
$$\psi_b(x) = \frac{\psi(bx)}{\psi(b)}, \quad \psi_{e-b}(x) = \frac{\psi((e-b)x)}{1-\psi(b)} \quad {\rm for}\quad x\in J(a_0,e).$$
This gives the convex combination
$$\psi = \psi(b) \psi_b + (1-\psi(b) )\psi_{e-b}$$
and therefore $\psi=\psi_b$, which gives $\psi(bx) = \psi(b)\psi(x)$ for all $x\in J(a_0,e)$ and
in particular $\psi(a_0^2) =\psi (a_0)^2$.

It follows that  $\psi((\lambda_0 v)^2) = \psi ((e-a_0)^2)
= \psi (e-2a_0 +a_0^2) =(1-\psi(a_0))^2 \geq 0$ for each pure state $\psi\in S_{a_0}$,
and hence $\vp(v^2) \geq 0$ for all states $\vp\in S_{a_0}$, by the Krein-Milman theorem.
As each state of $V$ restricts to a state of $J(a_0,e)$, we have shown $f(v^2) \geq 0$ for all states $f$ of $V$ and
hence $v^2\in \overline \Omega$ by (\ref{po}). This proves the second inclusion in (\ref{equal}).

The preceding arguments also reveal that  $\|v\|_e^2 =\|v^2\|_e$ since $\psi(v^2)=\psi(v)^2$ for all
pure states of $J(a_0,e)$ and $\|v\|_e$ is the supremum $\sup\{|\psi(x)|\}$, taken over all
pure states $\psi$ in $S_{a_0}$. Since $v\in V$ was arbitrary, we have shown $\|x^2\|_e =\|x\|_e^2$
for all $x\in V$.

In (\ref{banach}), we now actually have
$$\|xy\|_e \leq \|x\|_e\|y\|_e \qquad (x,y\in V).$$
This follows from the fact that
the map $(x,y)\in V^2 \mapsto f(xy)\in \mathbb{R}$ is a positive
semi-definite symmetric bilinear form, for each state
$f \in S_e$, and hence the Schwarz inequality gives
$$|f(xy)| ^2 \leq f(x^2)f(y^2) \leq \|x^2\|_e\|y^2\|_e = \|x\|_e^2\|y\|_e^2$$
and $\|xy\|_e = \sup\{|f(xy)| : f\in S_e\}\leq \|x\|_e\|y\|_e$.

Let $a\in V$. For all $x, y \in J(a,e)$,  the inequality $0\leq x^2 \leq x^2+y^2$
implies $\|x^2\|_e \leq \|x^2+y^2\|_e$. Therefore
we have shown that $(J(a,e),\|\cdot\|_e)$ is an associative JB-algebra, which can be identified with the
algebra $C(\mathcal{S},\mathbb{R})$ of real continuous functions on a compact Hausdorff space $\mathcal{S}$ \cite[Theorem 3.2.2]{stormer}.
Equipped with the injective tensor norm $\|\cdot\|_{inj}$, the complexification $J(a,e)_c =C(\mathcal{S}, \mathbb{R}) \otimes \mathbb{C}$
identifies with the C*-algebra  $C(\mathcal{S},\mathbb{C})$ of complex continuous functions on $\mathcal{S}$ .

Equip the complexification $V_c = V \otimes \mathbb{C}$ of $(V, \|\cdot\|_e)$
 with the injective tensor norm $\|\cdot\|_{inj}$. Then, for  $a\in V$, the remarks at the end of Section \ref{jordanorder}
 imply that
  the numerical range $\textsf{V}(L_{a^2})$ of the
  left multiplication operator  $L_{a^2}: V_c \longrightarrow V_c$
  is contained in
 $$\textsf{v}(a^2)=  \{f(a^2): f\in V_c^* \mbox{ satisfies } \|f\|=1=f(e)\}$$
 where each $f$ restricts to a state of the C*-algebra $J(a,e)_c = C(\mathcal{S},\mathbb{C})$.
Since $a^2 \in J(a,e) \cap \overline\Omega \subset C(\mathcal{S},\mathbb{R})$,
 we have $f(a^2) \geq 0$ and in particular, $\textsf{V}(L_{a^2}) \subset \textsf{v}(a^2)
 \subset  \mathbb{R}$.   Hence
the operator $L_{a^2}$ is hermitian  in $L(V_c)$ and as in (\ref{eta}), the linear operator
\begin{equation}\label{etaa}
 S\in L(V_c) \mapsto [L_{a^2}, S] = L_{a^2}S - SL_{a^2} \in L(V_c)
\end{equation}
is hermitian.

We are now equipped to prove the Jordan identity. Indeed,  we have
$$[L_{a^2}, [L_{a^2}, L_a]]=0.$$
by (\ref{234}) and as before, applying
Lemma \ref{95} to the hermitian operator  in (\ref{etaa}) yields
$$ \|[L_{a^2}, L_a]\|^2 \leq
 4\|L_a\|\|[L_{a^2}, [L_{a^2}, L_a]]\|=0 $$
and therefore  $ [L_{a^2}, L_a]=0$, proving the Jordan identity in $V$.

It remains to show that $(V, \|\cdot\|_e)$ is a JB-algebra and $\overline\Omega =
\{x^2: x\in V\}$.
To show the former, it suffices to prove
$$-e\leq a \leq e \Rightarrow 0\leq a^2 \leq e \qquad (a\in V)$$
by \cite[Proposition 3.1.6]{stormer}. Let $-e\leq a\leq e$. We have already shown
$a^2 \in \overline \Omega$. Since $e\pm a \in  \overline\Omega \cap J(a,e)$ and all pure states
 of $J(a,e) \approx C(\mathcal{S}, \mathbb{R})$ are multiplicative,  we have
$$\psi(e- a^2) = \psi((e+a)(e-a)) = \psi(e+a)\psi(e-a)\geq 0$$
for all pure states $\psi $ of $J(a,e)$,
which implies $\vp(e-a^2) \geq 0$ for all states $\vp$ of $J(a,e)$, by the Krein-Milman theorem. Hence
$e-a^2 \in \overline \Omega$ since each state of $V$ restricts to a state of $J(a,e)$.
This proves that $(V,\|\cdot\|_e)$ is a JB-algebra.
It follows that
$\{x^2: x\in V\}$ is closed and coincides with $\overline\Omega$, by (\ref{equal}).
\end{proof}

\begin{xrem}
The proof of Theorem \ref{main} reveals that  condition (iii) in the theorem is equivalent to
$\Omega$ being a normal linearly homogeneous Finsler symmetric cone in the tangent norm $\tau$ defined
in (\ref{tau}). However, (iii)  can also be equivalent to $\Omega$ being a normal linearly homogeneous
Finsler symmetric cone in another $G(\Omega)$-invariant tangent norm. For instance,
the other tangent norm can be the Riemannian metric given in  Example \ref{g} below.
\end{xrem}

\begin{exa}\label{spin}
Let $H$ be a real Hilbert space with norm $\|\cdot\|$ and inner product
$\langle\cdot,\cdot\rangle$.
The Hilbert space direct sum
$H\oplus\mathbb{R}$ , with inner product  $\ll \cdot,\cdot\gg$,
is a JH-algebra with identity $e=0\oplus 1$ and the Jordan product
$$(a\oplus \alpha)(b\oplus \beta) := (\beta a +\alpha b) \oplus
(\langle a, b\rangle +\alpha \beta).$$
We have
$\{x^2: x\in H \oplus \mathbb{R}\} = \{ a \oplus \alpha: \alpha \geq \|a\|\}$. Its interior $\Omega $
is linearly homogeneous \cite[Lemma 2.3.17]{book} and  a Riemannian symmetric space in the metric
$$g_p(u,v) = \, \ll\{p^{-1}, u, p^{-1}\}, v\gg \qquad (p\in \Omega, u, v \in H \oplus \mathbb{R})$$
\cite[Theorem 2.3.19]{book}
where $\{p^{-1}, u, p^{-1}\}$ denotes the Jordan triple product.

One can define an equivalent norm $\|\cdot\|_s$ on $H \oplus \mathbb{R}$
by
$$\|a\oplus \alpha\|_s = \|a\| + |\alpha|.$$ When $H \oplus \mathbb{R}$ is equipped with this norm,
it becomes a JB-algebra and is called a {\it spin factor}, where $\|\cdot\|_s$ is the order-unit norm
induced by $e$.  In this setting, $\Omega$ is a linearly homogeneous
Finsler symmetric cone with the tangent norm $\tau$ in (\ref{tau}), which differs from $g$.  We have
$$\tau(e, a\oplus \alpha) =\|a\oplus \alpha\|_e = \|a\oplus \alpha\|_s = \|a\| + |\alpha|$$  whereas
 $g_e( a\oplus \alpha, a\oplus \alpha)^{1/2} = \sqrt{\|a\|^2 +|\alpha|^2}$.
\end{exa}

The class of JB-algebras include the unital JH-algebras. Indeed,
unital JH-algebras have been classified in \cite[Section 3]{chu1}, they are of the form
\begin{equation}\label{dsum}
A_1 \oplus \cdots \oplus A_n \qquad (n\in \mathbb{N})
\end{equation}
where each summand $A_j$ is either a finite dimensional unital JH-algebra or
of the form
$H \oplus \mathbb{R}$,
and the direct sum in (\ref{dsum}) is equipped with coordinatewise Jordan product and the $\ell_2$-norm
$$\|a_1 \oplus \cdots \oplus a_n\|_2 := (\|a_1\|^2+ \cdots + \|a_n\|^2)^{1/2}.$$
When the direct sum is equipped
with the $\ell_\infty$-norm
$$\|a_1 \oplus \cdots \oplus a_n\|_\infty := \sup\{\|a_1\|, \ldots, \|a_n\|\},$$ it becomes a JB-algebra.
Finite dimensional unital JH-algebras are exactly the class of finite dimensional formally real
Jordan algebras, which have been classified in \cite{jvw}.

\begin{corollary}\label{fin} Let $\Omega$ be a proper open cone in a real Hilbert space $V$, with closure $\overline\Omega$.
The following conditions are equivalent.
\begin{enumerate}
\item[\rm(i)] $\Omega$ is a normal linearly homogeneous Finsler symmetric cone.
\item[\rm (ii)] $\Omega$ is a linearly homogeneous self-dual cone.
\item[\rm (iii)] $V$ is a unital JH-algebra in an equivalent norm and $\overline\Omega = \{a^2: a\in V\}.$
\end{enumerate}
\end{corollary}
\begin{proof}
(ii) $\Rightarrow$ (iii). This has been proved in \cite{chu}. In fact, condition (ii) entails
a decomposition  $\frak g(\Omega) = \frak k_1 \oplus \frak p_1$ and the evaluation
map $X\in \frak p_1 \mapsto X(e)\in V$ induces an algebra product in $V$, as in (\ref{jproduct}).
 One can use the argument
in the proof of  Theorem \ref{main} to derive the Jordan identity in place of the one given in \cite{chu}.

(iii) $\Rightarrow$ (ii). This has been proved in \cite[Lemma 2.3.17]{book}.

(iii) $\Rightarrow$ (i). This follows from Theorem \ref{main} since $V$ is a unital JB-algebra
in an equivalent norm by the preceding remark.

(i) $\Rightarrow$ (iii). By Theorem \ref{main}, $V$ is a unital JB-algebra  in an equivalent norm
and $\overline\Omega = \{a^2: a\in V\}.$
Since $V$ is a Hilbert space, it is a reflexive JB-algebra and by \cite[Corollary 3.3.6]{chu1},
$V$ is an $\ell_\infty$-sum of a finite number of finite dimensional formally real
Jordan algebras or spin factors, or both. Hence $V$
is a unital JH-algebra in an equivalent norm.
\end{proof}

\begin{xrem} It follows from the preceding corollary that one can view linearly homogeneous
Finsler symmetric cones as a generalisation
of linearly homogeneous self-dual cones to the setting of Banach spaces.
\end{xrem}

\begin{exa}\label{g}
A proper open cone $\Omega$ in a finite dimensional Euclidean space $\mathbb{R}^n$,
with inner product $\langle\cdot,\cdot\rangle$ and Euclidean measure $dy$, can be equipped with a
canonical $G(\Omega)$-invariant Riemannian metric \cite{vin}
\begin{equation}\label{conemetric}
g =  \frac{\partial ^2 \log \vp}{\partial x^i \partial x^j} dx^i dx^j
\end{equation}
where $\vp$ is the characteristic function of $\Omega$ defined by
$$\vp(x) = \int_{\Omega^*} \exp -\langle x,y\rangle dy \qquad (x\in \Omega).$$
The tangent norm $\nu$ defined by $g$  is not the same as
$\tau$ in (\ref{tau}).
 It has been shown in \cite{tsuji} and \cite{shima} that a linearly homogeneous cone $\Omega$ in $\mathbb{R}^n$  is
self-dual if $(\Omega,g)$ is a symmetric space. We see that  (i) $\Rightarrow$ (ii) in Corollary \ref{fin} provides an alternative proof
of this result, as well as extends it to infinite dimension.
\end{exa}

\bigskip
\footnotesize
\noindent\textit{Acknowledgment.}
This research was partly supported by EPSRC (UK) (grant no. EP/R044228/1).

\end{document}